\documentclass[12pt]{amsart} 
\usepackage{amsmath, amssymb, amsthm}
\usepackage{hyperref} 
\usepackage{geometry} 
\usepackage{natbib}
\usepackage{color}
\usepackage{hyperref}
\newtheorem{thm}{Theorem}[section]

\newtheorem{lem}[thm]{Lemma}

\theoremstyle{definition}
\newtheorem{defn}{Definition}[section]
\newtheorem{example}[defn]{Example}
\newtheorem{rem}[defn]{Remark}

{\qed\bigskip}

\newcounter{alphabet}

\title{Area Theorems and Quasiconformal Extensions of Harmonic Mappings with a Pole}
\author[Z. Chen]{Zhijun Chen}
\address{School of Statistics,
	University of International Business and Economics, No.~10, Huixin
	Dongjie, Chaoyang District, Beijing 100029, China}
\email{}
\author[L.-M.~Wang]{Li-Mei Wang}
\address{School of Statistics,
	University of International Business and Economics, No.~10, Huixin
	Dongjie, Chaoyang District, Beijing 100029, China}
\email{wangmabel@163.com}
\thanks{The authors were supported by a grant of University of International Business and Economics (No.78210418).
}
\begin{document}
	\maketitle

	\begin{abstract}
		In this paper, we study the class $\Sigma_{H}^{k}(p)$ of sense-preserving univalent harmonic mappings in the unit disk $\mathbb{D}$ that possess a simple pole at $p\in[0,1)$ and admit a $k$-quasiconformal extension to the extended complex plane for $k\in[0,1)$. 
		In 2024, Bhowmik and Satpati established an area theorem and derived a sufficient condition for the $k$-quasiconformal extension of harmonic mappings belonging to  $\Sigma_{H}^{k}(p)$  without logarithmic terms.
		Motivated by their work, we investigate the corresponding problem when a logarithmic singularity is present. Our main contributions are two-fold: we first prove a generalized area theorem for all mappings in $\Sigma_{H}^{k}(p)$; we then obtain a sufficient condition for sense-preserving univalent harmonic mappings in $\mathbb{D}$ to admit explicit $k$-quasiconformal extensions. 
		These results extend the aforementioned work to the setting where logarithmic singularities are allowed.
	\end{abstract}
	
		%
		%
		%
	
	\section{Introduction and main results}
	
	Let $\mathbb{D} = \{z \in \mathbb{C} : |z| < 1\}$  and $\overline{\mathbb{D}}$ denote the open and closed 
	unit disk in the complex plane $\mathbb{C}$, and $\partial\mathbb{D}$ the boundary of $\mathbb{D}$.
	Denote the extended complex plane by $\hat{\mathbb{C}}=\mathbb{C}\cup\{\infty\}$. 
	Let $\mathbb{D}^{*}=\{z\in\hat{\mathbb{C}}:|z|>1\}$ be the exterior of the unit disk and $\overline{\mathbb{D}^{*}}$ the closure of $\mathbb{D}^{*}$. 
	
	A complex-valued function $f$ is said to be \textit{harmonic} in $\mathbb{D}$ if both its real and imaginary parts are harmonic. 
	As is well known, a harmonic mapping $f$ in the unit disk can be represented by $f=h+\bar{g}$,
	where both $h$ and $g$ are analytic in $\mathbb{D}$.
	The Jacobian of $f$ is defined by 
	\[
	J_f (z)= |h'(z)|^2 - |g'(z)|^2.
	\]
	The mapping $z\mapsto f(z)$ is locally univalent if $J_f (z)\not=0$ in $\mathbb{D}$. 
	The result of Lewy \cite{Lewy} confirmed that the converse is also true for harmonic mappings.
	Therefore,  $z\mapsto f(z)$ is locally univalent and sense-preserving if and only if its Jacobian $J_f (z)>0$ in $\mathbb{D}$. 
	The study of harmonic mappings is an active area in geometric function theory (see Duren \cite{duren}).

	In contrast to conformal mappings, harmonic mappings are not determined by their image domains. 
	Motivated by this fact, Hengartner and Schober in \cite{original} initiated the study of sense-preserving  univalent  harmonic mappings on $\mathbb{D}^{*}$ which fix $\infty$. 
	They proved that such mappings admit the representation
	$$
	f(z)=\alpha z+\beta \bar{z}+h\left(\frac1z\right)+\overline{g\left(\frac1z\right)}+A \log |z|
	$$
	where $0 \leq|\beta|<|\alpha|$ and  $h$ and $g$ are analytic in $\mathbb{D}$ and $A\in \mathbb{C}$.
	After an affine post-mapping to $f$, we may normalize $f$ as 
	\begin{align*}
		f(z)=z+h\left(\frac1z\right)+\overline{g\left(\frac1z\right)}+A\log|z|.
	\end{align*}
	Let $\Sigma_{H}'$ be the set of all the above normalized sense-preserving univalent harmonic mappings in $\mathbb{D}^{*}$.
	In the same paper, Hengartner et. al. proved that the family $\Sigma_{H}'$ is compact with respect to the topology of locally uniform convergence. In addition, if $f\in \Sigma_{H}'$, then $|A|\le 2$.
	
	Let $p\in [0,1)$. Through a M\"{o}bius transform $w=(1-pz)/(z-p)$,  
	we can change the problem from the family $\Sigma_{H}'$ in $\mathbb{D}^{*}$ to $\Sigma_{H}'(p)$ in $\mathbb{D}$,
	where  $\Sigma_{H}'(p)$  is the class of all sense-preserving univalent harmonic mappings $f$ in $\mathbb{D}$ that have a simple pole at the point $p$, and are re-normalized as:
	\begin{align}\label{001}
		f(z)=\frac{1}{z-p}+h(z)+\overline{g(z)}+A\log \left|\frac{1-pz}{z-p}\right|,\quad z\in \mathbb{D}\setminus\{p\},   
	\end{align}
	where 
	\begin{equation}\label{hg}
		h(z)=\sum_{n=1}^{\infty}a_{n}z^{n},\quad \text{and} \quad g(z)=\sum_{n=1}^{\infty}b_{n}z^{n},\ z\in \mathbb{D}. 
	\end{equation}

		The special subclass of $\Sigma_{H}'(p)$ with $A=0$ was investigated in \cite{original}, \cite{areatheorem} and \cite{qcextension}. 
	An important topic in the theory of harmonic mappings is the study of quasiconformal extensions. 
	We first introduce the definition of quasiconformal mappings and quasiconformal extensions as follows.
For detailed information on quasiconformal mapping, we refer to \cite{lehto}.
	\begin{defn}
		Let $\Omega \subset \widehat{\mathbb{C}}$ be a domain. A sense-preserving homeomorphism $F: \Omega \to \widehat{\mathbb{C}}$ is called \emph{$k$-quasiconformal} $($$k$-q.c. in short$)$ for $k \in [0, 1)$  if it has locally $L^2$-derivatives $($in the sense of distributions$)$ on $\Omega \setminus \{F^{-1}(\infty), \infty\}$ and satisfies
		\[
		\left| F_{\bar{z}} \right| \leq k \left| F_{z} \right| \quad \text{almost everywhere}.
		\]
		The quantity $\mu_F = F_{\bar{z}} / F_z$ is called the \emph{complex dilatation}$($or {\emph Beltrami coefficient$)$} of $F$.
	\end{defn}

	\begin{defn}
		Let $k \in [0,1)$. A univalent holomorphic  $($or meromorphic or harmonic$)$ function $f$ defined on a domain $\Omega \subset \widehat{\mathbb{C}}$ is called to \emph{admit a $k$-quasiconformal extension to $\widehat{\mathbb{C}}$} if there exists a $k$-quasiconformal mapping $F: \widehat{\mathbb{C}} \to \widehat{\mathbb{C}}$ such that $F|_\Omega = f$. 
        Furthermore, a holomorphic function $f: \Omega \to \mathbb{C}$ $(\Omega \subset \mathbb{C})$ is said to admit a $k$-quasiconformal extension to $\mathbb{C}$ if there exists a $k$-quasiconformal mapping $F: \mathbb{C} \to \mathbb{C}$ satisfying $F|_\Omega = f$.
	\end{defn}

	Let $\Sigma_{H}^{k}(p)$ consist of harmonic mappings in $\Sigma_{H}'(p)$ that admit $k$-quasiconformal extensions to $\hat{\mathbb{C}}$. 
	Mateljevi\'c studied the area theorem of harmonic mappings in $\Sigma_{H}^{k}(0)$ by using Dirichlet's principle.
    For the history of the classical Gronwall's area theorem and its generalizations, the reader is referred to 
    recent work \cite{areatheorem} due to Bhowmik and Satpati and the references therein.
	In the same paper, they provided the area theorem and coefficient estimates for harmonic mappings belonging to $\Sigma_{H}^{k}(p)$ in \cite{areatheorem}. 
	But they only considered the case which does not contain the logarithmic term. 
	Our first aim is to generalize their results and evaluate the coefficients for all harmonic mappings in $\Sigma_{H}^{k}(p)$ as follows.
	
	\begin{thm}\label{areatheorem}
		Let $0\le k <1,~0\le p<1$ and $f\in \Sigma_{H}^{k}(p)$ have expressions of the form \eqref{001}, then
		\begin{align}\label{coefficient}
			\sum_{n=1}^{\infty}n|a_{n}|^{2}\le k^{2}\Big[\frac{1}{(1-p^{2})^{2}}+2\operatorname{Re}\left(\sum_{n=1}^{\infty}np^{n-1}b_{n}\right)+\sum_{n=1}^{\infty}n|b_{n}|^{2}\Big]
		\end{align}
        and 
        \begin{equation}\label{eq:A}
        |A|< \frac{2k}{1-p^{2}}.
        \end{equation}
		Equality holds in inequality \eqref{coefficient} only for functions in $\Sigma_{H}^{k}(p)$ that are of the form 
		\begin{align}
			f(z)=\frac{1}{z-p}+\frac{cz}{1-pz}+cg(z)+\overline{g(z)}+A\log \left|\frac{1-pz}{z-p}\right|,\quad z\in \mathbb{D},
			\label{108}
		\end{align}
		where $g(z)$ is analytic in $\mathbb{D}$ and $c$ is a constant with $|c|=k$. Moreover, the $k$-quasiconformal extension of $f$ to $\overline{\mathbb{D}^*}$ is given by 
		\begin{align}
			\tilde{f}(z)=\frac{1}{z-p}+\frac{c}{\overline{z}-p}+cg\left(\frac{1}{\overline{z}}\right)+\overline{g}\left(\frac{1}{\overline{z}}\right),\quad z\in  \overline{\mathbb{D}^{*}}.
			\label{109}
		\end{align}
        
	\end{thm}
	
	\begin{rem} 
		The inequality \eqref{coefficient} is the same as in Theorem 1 of \cite{areatheorem}
		which means that the logarithmic term does not affect the coefficient estimate.
	\end{rem}


		Recently, sufficient conditions for harmonic mappings with a real pole to admit quasiconformal extensions were also obtained by Bhowmik and Satpati in \cite{qcextension} and Ma in \cite{Ma}. However, all their results do not cover mappings with logarithmic singularities yet. So, the second main result of this paper provides a sufficient condition for harmonic mappings in $\Sigma_{H}'(p)$ to admit $k$-quasiconformal extensions. Furthermore, the quasiconformal extension is also explicitly given  in the theorem as follows.
		\begin{thm}\label{extensiontheorem}
			Let $0\le k<1$ and $0\le p<1$. Let $f$ be a harmonic mapping in $\mathbb{D}$ of the form \eqref{001} that satisfies 
			\begin{equation}
				|h'(z)|+|g'(z)|\le \frac{k}{(1+p)^{2}}-\frac{p^{2}+1}{1-p^{2}}|A|,\quad z\in\mathbb{D},
				\label{106}
			\end{equation}
			then $f\in \Sigma_{H}^{k}(p)$. The quasiconformal extension of $f$ is given by
			\begin{align}
				F(z)=\begin{cases}
					f(z), &z\in \mathbb{D},\\
					\frac{1}{z-p}+h(\frac{1}{\overline{z}})+\overline{g}(\frac{1}{\overline{z}}),\quad &z\in \overline{\mathbb{D}^{*}}.
				\end{cases}
				\label{107}
			\end{align}
		\end{thm}
		
		\begin{rem} 
			Note that if $A=0$, Theorem \ref{extensiontheorem} reduces to Theorem 4 in \cite{qcextension}.
		\end{rem}
		
		The next two sections are devoted to proving Theorems \ref{areatheorem} and \ref{extensiontheorem}.

		\section{Proof of Theorem \ref{areatheorem}}
		
		In order to prove Theorem \ref{areatheorem}, we introduce some auxiliary lemmas.
		The first shows the area of the omitted set of a mapping in $\Sigma_{H}'(p)$
		which is the same as Lemma 1 of \cite{areatheorem} and thus demonstrates that the area is irrelevant to the logarithmic term.
		
		\begin{lem}\label{lemma1}
			Let $f \in \Sigma_{H}'(p)$ be given in the form \eqref{001}, then
			\begin{align}\label{eq:area}
				\mathrm{Area}(\widehat{\mathbb{C}} \setminus f(\mathbb{D})) = &\pi \Big[ \frac{1}{(1 - p^{2})^{2}} + 2 \operatorname{Re} \left( \sum_{n=1}^{\infty} n p^{n-1} b_{n} \right) + \sum_{n=1}^{\infty} n (|b_n|^{2} - |a_n|^{2}) \Big].
			\end{align}
		\end{lem}
		
		\begin{proof}
			We follow the strategy of proof of Lemma 1 in \cite{areatheorem} (see also  Theorem 3.8 in \cite{original} ).
			Let $P$ denote the area of the set omitted by $f$, that is, 
			$P = \mathrm{Area}(\widehat{\mathbb{C}} \setminus f(\mathbb{D}))$. 
			For the case $p=0$, we refer to [\citenum{original}, Thm.3.8] and find that
			\begin{align*}
				P=\pi\left[1+2\operatorname{Re}b_{1}+\sum_{n=1}^{\infty} n (|b_n|^{2} - |a_n|^{2})\right].
			\end{align*} 
			Otherwise, for each circle $C_r:~\{z : |z| = r\}$ with $0 < p < r < 1$, its image under $f$ is a simple closed curve. 
			With counterclockwise orientation, this curve bounds a domain of finite area, denoted by $P_r$. 
			First, we observe that
			$$
			\log\left|\frac{1-pz}{z-p}\right|=\log \left|\frac{r(1-pz)}{r^2-pz}\right|, \quad z\in C_r
			$$
			and 
			$$
			d\log\left|\frac{1-pz}{z-p}\right|
			=-\frac{1-p^2}{2}\left(\frac{dz}{(z-p)(1-pz)}+\frac{d\overline{z}}{(\overline{z}-p)(1-p\bar{z})}\right),
			\quad |z|<r.
			$$
			Using the expansion of $f$ and the above equations, we obtain by computations that
			\begin{align*}
				P_r&= -  \frac{1}{2i} \int_{C_r} \overline{f} \, df \\
				&= - \frac{1}{2i} \int_{C_r} \left( \frac{1}{\overline{z} - p} + \overline{h(z)} + g(z)+\overline{A}
				\log \left|\frac{1-pz}{z-p}\right|\right)\\
				& \quad  \left( - \frac{1}{(z - p)^2} dz + h'(z) dz + \overline{g'(z)} d\overline{z}+Ad\log \left|\frac{1-pz}{z-p}\right|\right) \\
				&= \frac{1}{2i} \left[ \int_{C_r} \frac{dz}{(\overline{z} - p)(z - p)^2} 
				+ \int_{C_r} \frac{gdz}{(z - p)^2}- \int_{C_r} \overline{h} h' dz\right. - \int_{C_r} \frac{\overline{g'} d\overline{z}}{\overline{z} - p} - \int_{C_r} g \overline{g'} d\overline{z}  \\
				&\quad-\overline{A} \int_{C_r} \log \left|\frac{r(1-pz)}{r^2-pz}\right|
				\left( - \frac{1}{(z - p)^2} dz + h' dz + \overline{g'} d\overline{z}\right) \\
				&\quad\left.-\frac{(1-p^2)A}{2}\int_{C_r}\left( \frac{1}{\overline{z} - p} + \overline{h} + g\right)\left(\frac{dz}{(z-p)(1-pz)}+\frac{d\overline{z}}{(\overline{z}-p)(1-p\bar{z})}\right)\right] \\
				&\quad \text{(rest of the terms are zero)} \\
				&=:  \pi \left[ \frac{r^2}{(r^2 - p^2)^2}  + \sum_{n=1}^{\infty} n p^{n-1} b_n - \sum_{n=1}^{\infty} n |a_n|^2 r^{2n} + \sum_{n=1}^{\infty} n p^{n-1} \overline{b}_n + \sum_{n=1}^{\infty} n |b_n|^2 r^{2n}\right] \\
				&\quad - \overline{A}I_{1}-\frac{(1-p^2)A}{2}I_{2}.
			\end{align*}
			Next we calculate the integrals $I_{1}$ and $I_{2}$. 
			Let $\tau(z)=\log(r(1-pz)/(r^2-pz))$ for simplicity.
			Observe that the function $z\mapsto \tau(z)$ is analytic in the disk surrounded by $C_r$ and thus 
			has the power series
			$$
			\tau(z)=\sum_{n=0}^{\infty}c_nz^n,\quad |z|<r.
			$$
			Thus after a fundamental computation, we deduce from Cauchy's integral formula that 
			\begin{align*}
				I_{1}&:=\frac{1}{2i}\int_{C_r}\log \left|\frac{r(1-pz)}{r^2-pz}\right|
				\left( - \frac{1}{(z - p)^2} dz + h' dz + \overline{g'} d\overline{z}\right)\\
				&=\frac{1}{4i}\int_{C_r}\left(\tau(z)+\tau(\bar{z})\right)
				\left( - \frac{1}{(z - p)^2} dz + h' dz + \overline{g'} d\overline{z}\right)\\
				&=\frac{1}{4i}\int_{C_r}\left(\tau(z)+\tau(\bar{z})\right)
				\frac{-dz}{(z - p)^2}+h'\tau(\bar{z})dz+\tau(z)\overline{g'} d\overline{z}\\
				&=\frac{\pi}{2}\left(-p\left(\tau(z)\right)'|_{z=p}-\sum_{n=1}^{\infty}(n+2)c_nr^{2n}
				+c_0+\sum_{n=1}^{\infty}na_nc_nr^{2n}-c_0-\sum_{n=1}^{\infty}nb_nc_nr^{2n}\right)\\
				&=\frac{\pi}{2}\left(\frac{p(1-r^2)}{(1-p^2)(r^2-p^2)}+\sum_{n=1}^{\infty}(na_n-nb_n-(n+2))c_nr^{2n}\right).
			\end{align*}
			Since $\tau(z)$ is uniformly convergent in $\mathbb{D}$ to $0$ as $r\to 1^{-}$,
			$\lim_{r\to1^-}I_1=0$.
			
			On the other hand, we find $1/(\overline{z} - p)=z/(r^2-pz) $ on $C_r$.
			Therefore by making use of the technique similar t0 $I_1$, we can 
			compute $I_2$ as follows 
			\begin{align*}
				I_{2}&:=\frac{1}{2i}\int_{C_r}\left( \frac{1}{\overline{z} - p} + \overline{h} + g\right)\left(\frac{dz}{(z-p)(1-pz)}+\frac{d\overline{z}}{(\overline{z}-p)(1-p\bar{z})}\right)\\
				&=\frac{1}{2i}\left(\int_{C_r}\left( \frac{z}{r^2 - pz} +  g\right)\frac{dz}{(z-p)(1-pz)}+
				\int_{C_r} \left(\frac{\overline{h}}{(\overline{z}-p)(1-p\bar{z})}+\frac{1}{(\overline{z}-p)^2(1-p\bar{z})}\right)d\overline{z}\right.\\
				&\quad \left. + \int_{C_r}\frac{\overline{h} dz}{(z-p)(1-pz)}
				+\int_{C_r}\frac{g d\overline{z}}{(\overline{z}-p)(1-p\bar{z})}\right)\\
				&=\pi\left(\frac{z}{(r^2-pz)(1-pz)}+\frac{g}{1-pz}
				-\overline{\frac{h}{1-pz}} -\left(\frac{1}{1-pz}\right)'\right)|_{z=p}\\
				&\quad\left. +\int_{C_r}\frac{\overline{hz} dz}{(r^2-p\bar{z})(1-pz)}
				-\overline{\int_{C_r}\frac{\overline{gz} dz}{(r^2-p\bar{z})(1-pz)}}\right)\\
				&=\pi\left(\frac{p}{(r^2-p^2)(1-p^2)}+\frac{g(p)}{1-p^2}-\frac{\overline{h(p)}}{1-p^2}-\frac{p}{(1-p^2)^2}
				+\frac{1}{p}\frac{\overline{hz} }{r^2-p\bar{z}}\big|_{z=pr^2}-
				\frac{1}{p}\overline{\frac{\overline{gz} }{r^2-p\bar{z}}}\big|_{z=pr^2}\right)\\
				&=\pi\left(\frac{p}{(r^2-p^2)(1-p^2)}+\frac{g(p)}{1-p^2}-\frac{\overline{h(p)}}{1-p^2}-\frac{p}{(1-p^2)^2}
				+\frac{\overline{h(pr^2)} }{1-p^2}-
				\frac{g(pr^2)}{1-p^2}\right).
		\end{align*}	
			Thus   $\lim_{r\to1^-}I_2=0$.
			
			By applying Green’s theorem and the limit behaviors of $I_1$ and $I_2$, 
			we conclude that 
			$$
			P = \displaystyle\lim_{r \to 1^{-}} P_r
			=\pi \left[ \frac{1}{(1 - p^2)^2}  + \sum_{n=1}^{\infty} n p^{n-1} b_n - \sum_{n=1}^{\infty} n |a_n|^2  + \sum_{n=1}^{\infty} n p^{n-1} \overline{b}_n + \sum_{n=1}^{\infty} n |b_n|^2 \right] 
			$$
			which is the claimed equation \eqref{eq:area}. 
		\end{proof}

		The next lemma is a part result of  Lemma 3.3 in \cite{Wang001}. 
		In order to state it, we first introduce the Gaussian hypergeometric function ${}_2F_1(a,b;c;x)$ which is defined by 
		\[
		{}_2F_1(a,b;c;x)
		= \sum_{n=0}^{\infty}
		\frac{(a)_n (b)_n}{(c)_n\, n!}\, x^n, \quad |x|<1,
		\]
		where $(a)_n$ denotes the Pochhammer symbol given by $(a)_0=1$ and
		\[
		(a)_n = a(a+1)\cdots(a+n-1), \quad n\ge1.
		\]
		For basic properties of hypergeometric functions, see \cite{AAR}.
		\begin{lem}
			\label{Wang}$($\cite[ Lem. 3.3]{Wang001}$)$
			Let $\alpha, ~\beta \ge 0$. Then
			\begin{align*}
				\frac{1}{2\pi}\int_{0}^{2\pi}
				\frac{d\theta}{|1 - z e^{i\theta}|^{2\beta}}
				= (1 - |z|^2)^{1-2\beta}
				\, {}_2F_1(1-\beta,\,1-\beta;\,1;\,|z|^2),
				\quad z \in \mathbb{D}.
			\end{align*}

		\end{lem} 
		Let $Q(z)=u(x, y)+i v(x, y)$ be a complex valued $C^1$-function defined in $\mathbb{D}$, 
		where $z=x+i y$.
		The \emph{Dirichlet's integral} (or \emph{energy integral}) of $Q$ is defined as
		$$
		D[Q]=\iint_{\mathrm{D}}|\nabla Q|^2 d x d y=\iint_{\mathrm{D}}\left(|\nabla u|^2+|\nabla v|^2\right) d x d y,
		$$
		where $\nabla u=u_x+i u_y$ and $\nabla v=v_x+i v_y$. Dirichlet's principle basically states that, if we can fit a harmonic mapping to a given $C^1$-function on the boundary of $\mathbb{D}$, then the energy integral will be minimum for the harmonic mapping inside the disk $\mathbb{D}$.
		
		\begin{lem}\label{dirichlet}$($\cite{dirichlet}, Dirichlet's principle$)$
			Let $R$ and $Q$ be $C^1$-functions in $\overline{\mathbb{D}}$ and $Q$ harmonic in $\mathbb{D}$ such that $Q = R$ on $\partial \mathbb{D}$. 
			If $D[R]<\infty$, then
			\[
			D[Q] \leq D[R].
			\]
			Equality holds if and only if $Q \equiv R$ in $\mathbb{D}$.
		\end{lem}
		Now we are ready to prove Theorem \ref{areatheorem}.
		

		\begin{proof}[Proof of Theorem \ref{areatheorem}]
			Since $f \in \Sigma_H^k(p)$, it follows that $f$ is harmonic in $\mathbb{D}$ and extends to a homeomorphism of $\widehat{\mathbb{C}}$ onto itself which is $k$-q.c. in $\mathbb{D}^* $. 
			Denote $l(z)=\sum_{n=1}^{\infty}\overline{b_{n}}z^{n}=\overline{g(\bar{z})}$ for simplicity.
			Let $\phi(\zeta) = f(1/\zeta)$, $\zeta \in \mathbb{\hat{C}}$, so that it can  be written as
			\[
			\phi(\zeta) = \frac{\zeta}{1 - p\zeta} + h\left(\frac{1}{\zeta}\right) + l\left(\frac{1}{\overline{\zeta}}\right)+A\log\left|\frac{\zeta-p}{1-p\zeta}\right|, \quad \zeta \in \mathbb{D}^*.
			\]
			Motivated by the expression of $\phi$, we define another function $F$ in $\overline{\mathbb{D}}$ by
			\begin{equation}
				F(\zeta) = \frac{\zeta}{1 - p\zeta} + h(\overline{\zeta}) + l(\zeta), \quad \zeta \in \overline{\mathbb{D}}.
				\label{102}
			\end{equation}
            The main idea of the following proof is from \cite{areatheorem}.
            For the completeness, we give the detail.
			It is easy to verify that $F$ is continuous in $\overline{\mathbb{D}}$, harmonic on $\mathbb{D}$ (since $F_{\zeta\overline{\zeta}} = 0$), and that $F \equiv \phi$ on $\partial \mathbb{D}$. 
			We wish to apply Lemma \ref{dirichlet} to $F$ and $\phi$. Therefore, we first need to show that $D[\phi]$ is finite in $\mathbb{D}$. Since $f \in \Sigma_H^k(p)$, the function $\phi$ is harmonic in $\mathbb{D}^*$ and $k$-q.c. in $\mathbb{D}$, and hence it has locally $L^2$-derivatives in $\mathbb{D}$. 
			Writing $\phi(\zeta) = u(\xi,\eta) + i v(\xi,\eta)$, where $\zeta = \xi + i\eta$, we compute
			\[
			\phi_\zeta = \frac{(u_\xi + v_\eta) + i(v_\xi - u_\eta)}{2}, 
			\qquad
			\phi_{\overline{\zeta}} = \frac{(u_\xi - v_\eta) + i(v_\xi + u_\eta)}{2}.
			\]
			Hence the Dirichlet integral of $\phi$ is 
			\begin{equation}\label{eq:D}
				D[\phi] = \iint_{\mathbb{D}} |\nabla \phi|^2 \, d\xi d\eta 
				= 2 \iint_{\mathbb{D}} \left( |\phi_\zeta|^2 + |\phi_{\overline{\zeta}}|^2 \right) d\xi d\eta 
				= 4\alpha - 2P,
			\end{equation}
			where
			\[
			\alpha = \iint_{\mathbb{D}} |\phi_\zeta|^2 \, d\xi d\eta,
			\qquad
			P = \iint_{\mathbb{D}} \left( |\phi_\zeta|^2 - |\phi_{\overline{\zeta}}|^2 \right) d\xi d\eta.
			\]
			Since $\phi$ is $k$-q.c. (being the composition of $f$ with a conformal map) in $\mathbb{D}$,
			we have $|\phi_{\overline{\zeta}}| \leq k |\phi_\zeta|$ in $\mathbb{D}$
			which further implies $P \geq (1 - k^2)\alpha.$
			Thus we infer from the equation \eqref{eq:D} that
			\begin{equation}
				D[\phi] \leq 2P \left( \frac{1 + k^2}{1 - k^2} \right).
				\label{100}
			\end{equation}
			Now as $f$ is $k$-q.c., $f(\partial \mathbb{D})$ has area zero.
			We derive from $\phi(\zeta)=f(1/\zeta)$ that
			\begin{equation*}
				P=\iint_{\mathbb{D}}J_{\phi}(\zeta)d\xi d\eta =\operatorname{Area}(\phi(\mathbb{D}))=\operatorname{Area}(f(\mathbb{D}^{*})).
			\end{equation*}
			Thus it follows from Lemma \ref{lemma1} and \eqref{100} that
			\begin{equation}\label{eq:D-esti}
				D[\phi]\le  2\pi  \frac{1 + k^2}{1 - k^2}  \Big[ \frac{1}{(1 - p^{2})^{2}} + 2 \operatorname{Re} \left( \sum_{n=1}^{\infty} n p^{n-1} b_{n} \right) + \sum_{n=1}^{\infty} n (|b_n|^{2} - |a_n|^{2}) \Big]
			\end{equation}
			which is finite in $\mathbb{D}$.
			Since $F$ and $\phi$ satisfy the conditions of Lemma \ref{dirichlet}, we have 
			\begin{equation}
				D[F]\le D[\phi].
				\label{101}
			\end{equation} 
			
			We proceed to compute the Dirichlet integral $D[F]$ of the harmonic mapping $F$.
			Since $F$ has the expansion \eqref{102}, we have
			\begin{align}\label{eq:dirivatives}
				F_{\zeta}=\frac{1}{(1-p\zeta)^{2}}+\sum_{n=1}^{\infty}n\overline{b_{n}}\zeta^{n-1}\quad 
				\text{and} \quad 
				F_{\bar{\zeta}}=\sum_{n=1}^{\infty}na_{n}\bar{\zeta}^{n-1}.  
			\end{align}
			Letting $\beta$ be 2 in Lemma \ref{Wang}, we obtain the value of integral:
			\begin{align*}
				\iint_{\mathbb{D}}\frac{1}{|1-p\zeta|^{4}}d\sigma&=\int_{0}^{1}rdr\int_{0}^{2\pi}\frac{d\theta}{|1-pre^{i\theta}|^{4}}\\
				&=2\pi \int_{0}^{1}{}_{2}F_{1}(-1,-1;1;(pr)^{2})(1-(pr)^{2})^{-3}rdr\\
				&=2\pi \int_{0}^{1}\frac{1+p^{2}r^{2}}{(1-p^{2}r^{2})^{3}}rdr\\
				&=\frac{\pi}{(1-p^{2})^{2}}.
			\end{align*}
			Then by virtue of the forms in \eqref{eq:dirivatives} and the above integral,
			it is easy to find that
			\begin{align}
				\iint_{\mathbb{D}}|F_{\zeta}|^{2}d\sigma =& \pi\Big[\frac{1}{(1-p^{2})^{2}}+2\operatorname{Re}\Big(\sum_{n=1}^{\infty}np^{n-1}b_{n}\Big)+\sum_{n=1}^{\infty}n|b_{n}|^{2}\Big]
				\label{103}
			\end{align}
			and
			\begin{align}
				\iint_{\mathbb{D}}|F_{\overline{\zeta}}|^{2}d\sigma = \pi\sum_{n=1}^{\infty}n|a_{n}|^{2}.
				\label{104}
			\end{align}
			From \eqref{103} and \eqref{104}, we have
			\begin{align*}
				D[F]=2\iint_{\mathbb{D}}(|F_{\zeta}|^{2}+|F_{\overline{\zeta}}|^{2})d\sigma
				=2\pi\Big[\frac{1}{(1-p^{2})^{2}}+2\operatorname{Re}\Big(\sum_{n=1}^{\infty}np^{n-1}b_{n}\Big)+\sum_{n=1}^{\infty}n(|b_{n}|^{2}+|a_{n}|^{2})\Big].
			\end{align*}
			Then using the former value and the inequalities \eqref{eq:D-esti} and \eqref{101},
			we arrive at the required inequality 
			of this theorem.
			
			Next we consider the equality case. If equality holds in \eqref{coefficient}, then equality must hold in  both \eqref{100} and \eqref{101}. By Lemma \ref{dirichlet} equality in \eqref{101} implies that $F\equiv \phi$ in $\mathbb{D}$. Equality in \eqref{100} yields that $|\phi_{\overline{\zeta}}|=k|\phi_{\zeta}|$ in $\mathbb{D}$ and hence $|F_{\overline{\zeta}}|=k|F_{\zeta}|$. Thus by making use of the expression of $F$ in \eqref{102}, we find that
			\begin{align*}
				\left|h'(\overline{\zeta})\right|
				= k \left| \frac{1}{(1 - p\zeta)^2} + l'(\zeta) \right|,
				\quad \zeta \in \mathbb{D},
			\end{align*}
			which can be written as $|G(\zeta)| = k$, where
			\[
			G(\zeta) = \frac{\overline{h'(\overline{\zeta})}}{\dfrac{1}{(1 - p\zeta)^2} + l'(\zeta)}.
			\]
			
			We now show that $G$ is analytic in $\mathbb{D}$. From the above relation, if $\zeta_0$ is a zero of finite order of the denominator of $G$, then it is also a zero of the numerator of the same order. Hence $G$ is analytic in $\mathbb{D}$ with constant modulus, which implies that $G(\zeta) = \overline{c}$, where $c$ is a constant with $|c| = k$. Therefore, we have
			\[
			\overline{h'(\overline{\zeta})} = \frac{\overline{c}}{(1 - p\zeta)^2} + \overline{c}\,l'(\zeta).
			\]
			Since $h(0) = l(0) = 0$, integrating both sides of the above equation along any curve in $\mathbb{D}$ joining $0$ to $\zeta$, we have
			\[
			h(\overline{\zeta}) = \frac{c\,\overline{\zeta}}{1 - p\overline{\zeta}} + c\,\overline{l(\zeta)}.
			\]
			Inserting it into the expression \eqref{102} of $F(\zeta)$ above, we obtain
			\begin{align*}
				F(\zeta)=\frac{\zeta}{1-p\zeta}+l(\zeta)+\frac{c\,\overline{\zeta}}{1 - p\overline{\zeta}}
				+ c\overline{l(\zeta)}.
			\end{align*}
			At last, we can define $\hat{f}$ as
			\begin{align*}
				\hat{f}(z)=
				\begin{cases}
					f(z),\quad z\in\mathbb{D},\\
					\tilde{f}(z),\quad z\in\overline{\mathbb{D}^{*}},
				\end{cases}
			\end{align*}
			where $f(z)$ and $\tilde{f}(z)=\phi(1/z)=F(1/z)$ are given in \eqref{108} and \eqref{109}.
              
              We finally show the proof of  inequality \eqref{eq:A}.
              A simple calculate yields for $z\in\mathbb{D}$ that 
    \begin{align}\label{eq:w}
        \omega(z):=\frac{\overline{f_{\overline{z}}}}{f_{z}}=\frac{g'(z)-\frac{\overline{A}}{2(z-p)}-\frac{\overline{A}p}{2(1-pz)}}{-\frac{1}{(z-p)^{2}}+h'(z)-\frac{A}{2(z-p)}-\frac{Ap}{2(1-pz)}}=\frac{\overline{A}(z-p)}{2}+O((z-p)^{2}),
    \end{align}
    which is analytic in $\mathbb{D}$ and satisfy $|w(z)|\leq k$ for $z\in\mathbb{D}$ since $f\in \Sigma_{H}^{k}(p)$.
    Then we choose M\"{o}bius transformation $\phi_{p}$ to translate $p$ to the origin, and define the composition $F(z)=\omega(\phi_{p}^{-1}(z))$. Thus, the estimation of $|A|$ can be given by using the Schwarz Lemma of $F$
    \begin{align*}
        k\ge |F'(0)|=|\omega'(p)\cdot (\phi_{p}^{-1})'(0)|=\frac{|A|(1-p^{2})}{2}.
    \end{align*}
    Now we will prove by contradiction that the equality cannot hold. This equality implies that 
    $F(z)=\omega(\phi_{p}^{-1}(z))=ke^{i\theta}z$, for some $\theta\in[0,2\pi)$ and then we obtain
    \begin{align*}
        \omega(z)=ke^{i\theta}\frac{z-p}{1-pz}\quad \text{and}\quad ke^{i\theta}=\frac{\overline{A}(1-p^{2})}{2},
    \end{align*}
    by observing the first term of $w$ in \eqref{eq:w}.
    Substituting $\omega(z)$ into the expression \eqref{eq:w}, we derive the following equation
    \begin{align*}
        &g'(z)-\frac{\overline{A}(1-p^{2})}{2}h'(z)\\
        &=\frac{\overline{A}(1+p(1-p^{2})-(1+p-p^{2})z)}{2(z-p)(1-pz)}+\frac{Ap-|A|^{2}(1-p^{2})}{2(1-pz)}-\frac{|A|^{2}p(1-p^{2})(z-p)}{4(1-pz)^{2}}.
    \end{align*}
    Since the left-hand side is analytic in $\mathbb{D}$, the first term on the right-hand side of the former equation must satisfy the following conditions
    \begin{align*}
        1+p(1-p^{2})-(1+p-p^{2})z\Big|_{z=p}=0,\ \text{i.e.} \ p=1
    \end{align*}
    This contradicts the assumption $p\in [0,1)$.

			The proof is complete.

		\end{proof}

		\section{Proof of Theorem \ref{extensiontheorem}}

		In this section, we show the proof of Theorem \ref{extensiontheorem}.
		
		\begin{proof}[Proof of Theorem \ref{extensiontheorem}]
			If $k=0$, then $A=0$ and $f=1/(z-p)+c$ in $\hat{\mathbb{C}}$ is conformal, where $c$ is constant.
			Thus, $f$ has a trivial $0$-q.c. extension in $\hat{\mathbb{C}}$. 
			For $0<k<1$, we now verify that $f$ admits a $k$-q.c. extension to $\hat{\mathbb{C}}$.
			
			First, we show that $f$ is univalent on $\mathbb{D}$. 
			For $z_1,z_2\in\mathbb{D}$ with $z_{1} \neq z_{2}$, it follows from the assumption \eqref{106} that
			\begin{align}\label{110}
				|h(z_1)-h(z_2)+\overline{g(z_1)-g(z_2)}|&=\left|\int_{[z_{2},z_{1}]}h'(z)dz+\overline{g'(z)dz}\right|\\
				&\le\int_{[z_{2},z_{1}]}(|h'(z)|+|g'(z)|)|dz|\nonumber\\
				&\le\Big(\frac{k}{(1+p)^2}-|A|\Big)|z_1-z_2|.\nonumber
			\end{align}
			Thus $f(z)-1/(z-p)+A\log |(1-pz)/(z-p)|$ can be continuously extended to $\overline{ \mathbb{D}}$. The above inequality \eqref{110} still holds in $\overline{ \mathbb{D}}$, which implies that $f$ has a continuous extension to $\overline{ \mathbb{D}}$. For $z_{1}, z_{2}\in \overline{ \mathbb{D}}\setminus\{p\}$, without loss of generality, we suppose $|(1-pz_1)/(z_{1}-p)|\ge |(1-pz_1)/(z_{2}-p)|$. Then a computation yields
			\begin{equation}\label{eq:crossratio}
				1\le \left|\frac{1-pz_1}{z_{1}-p}\frac{z_{2}-p}{1-pz_2}\right|
				\le 1+(1-p^2)\frac{|z_{2}-z_{1}|}{|(z_{1}-p)(1-pz_2)|}
			\end{equation}
			which further implies 
			\begin{align*}
				&\left|\frac{1}{z_1-p}-\frac{1}{z_2-p}-A\log\frac{|z_{1}-p|}{|z_{2}-p|}+A\log\frac{|1-pz_{1}|}{|1-pz_{2}|}\right|\\
				\ge& \frac{|z_{2}-z_{1}|}{|z_{1}-p|\cdot|z_{2}-p|}-|A|\log\left( 1+(1-p^2)\frac{|z_{2}-z_{1}|}{|(z_{1}-p)(1-pz_2)|}\right)\\
				\ge&   \frac{|z_{2}-z_{1}|}{|z_{1}-p|\cdot|z_{2}-p|}-\frac{|A|(1-p^2)|z_{2}-z_{1}|}{|(z_{1}-p)(1-pz_2)|} \\
				\ge& \frac{|z_{2}-z_{1}|}{|z_{1}-p|}\Big(\frac{1}{|z_{2}-p|}-\frac{|A| (1-p^2)}{|1-pz_2|)}\Big)\\
				\ge&|z_{2}-z_{1}|\left(\frac{1}{(1+p)^{2}}-|A|\right),
			\end{align*}
			since $|A|<1/(1+p)^{2}$ by assumption.
			Hence, we infer from \eqref{110} and the above inequality that
			\begin{align}
				|f(z_1)-f(z_2)|
				\ge \frac{1-k}{(1+p)^2}|z_1-z_2|>0,
				\label{111}
			\end{align}
			which implies the univalence of $f$ in $\overline{ \mathbb{D}}$.
			
			Furthermore, we get from \eqref{111} that $f(\partial \mathbb{D})$ is a Jordan curve. 
			Moreover, since the inequalities \eqref{101} and \eqref{eq:crossratio} 
			are valid for $z_{1}, z_{2}\in\partial\mathbb{D}$ with $z_{1} \neq z_{2}$, 
			we have   
			\begin{align*}
				&|f(z_{1})-f(z_{2})|\\
				\le& |h(z_1)-h(z_2)+\overline{g(z_1)-g(z_2)}|+\left|\frac{1}{z_1-p}-\frac{1}{z_2-p}-A\log\frac{|z_{1}-p|}{|z_{2}-p|}+A\log\frac{|1-pz_{1}|}{|1-pz_{2}|}\right|\\
				\le& \Big(\frac{k}{(1+p)^2}-|A|\Big)|z_1-z_2|+\frac{|z_{2}-z_{1}|}{|z_{1}-p|}\left(\frac{1}{|z_{2}-p|}+\frac{|A|(1-p^{2})}{|1-pz_{2}|}\right)\\
				\le&  \Big(\frac{k}{(1+p)^2}-|A|\Big)|z_1-z_2|+\frac{|z_{2}-z_{1}|}{1-p}\left(\frac{1}{1-p}+|A|(1+p)\right)\\
				\le& |z_{1}-z_{2}|\left(\frac{k}{(1+p)^{2}}+\frac{1}{(1-p)^{2}}+\frac{2|A|p}{1-p}\right).
			\end{align*}
			Thus $f(\partial\mathbb{D})$ is a rectifiable Jordan curve of length at most  $2\pi[k/(1+p)^{2}+1/(1-p)^{2}+2p|A|/(1-p)]$.

			Then, we need to show that $f$ is sense-preserving in $\mathbb{D}$ and $|\mu_{f}|\le k$ in $\mathbb{D}\setminus\{p\}$. We derive from the expression of $f$ in \eqref{001} that 
			\begin{align*}
				f_{z}=-\frac{1}{(z-p)^{2}}+h'(z)-\frac{A}{2(z-p)}-\frac{Ap}{2(1-pz)}\text{ and }
				f_{\overline{z}}=\overline{g'(z)}-\frac{A}{2(\overline{z}-p)}-\frac{Ap}{2(1-p\overline{z})}.
			\end{align*}
			Then a computation yields for $z\in \mathbb{D}\setminus\{p\}$ that
			\begin{align*}
				|f_{z}|-|f_{\overline{z}}|&\ge \frac{1}{|z-p|}\left(\frac{1}{|z-p|}-|A|\right)-\frac{p|A|}{|1-pz|}-|h'(z)|-|g'(z)|\\
				&\ge\frac{1-k}{(1+p)^{2}}-|A|\left(\frac{1}{1+p}+\frac{p}{1-p}-\frac{1+p^{2}}{1-p^{2}}\right)\\
                &=\frac{1-k}{(1+p)^{2}}\\
                &> 0,
			\end{align*}
            as $1/|z-p|-|A|\ge 1/(1+p)-|A|$ by assumption.
			Thus $J_{f}(z)=|f_{z}|^{2}-|f_{\overline{z}}|^{2}>0$ holds for $z\in \mathbb{D}\setminus\{p\}$ and $J_{f}(p)=\infty$.
			So $f$ is sense-preserving in $\mathbb{D}$.
			Since sense-preserving harmonic mappings are necessarily open maps (see [\citenum{duren}, p.10]), it also follows that $f$ is a homeomorphism in $\overline{ \mathbb{D}}$.
			Moreover, in view of the inequality \eqref{106}, we find out for $z\in \mathbb{D}\setminus\{p\}$, 
			\begin{align*}
				|g'(z)|+k|h'(z)|&\le |g'(z)|+|h'(z)|\\
				&\le \frac{k}{(1+p)^{2}}-\frac{|A|(1+k)(1+p^{2})}{2(1-p^{2})}\\
				&\le \frac{k}{|z-p|^2}-\frac{(1+k)|A|}{2(1+p)}-\frac{(1+k)p|A|}{2(1-p)}\\
				&\le \frac{1}{|z-p|}\left(\frac{k}{|z-p|}-\frac{(1+k)|A|}{2}\right)-\frac{|A|p(k+1)}{2(1-p)}
			\end{align*}
            as $k/|z-p|-(1+k)|A|/2>0$ by assumption.
			Hence, the complex dilatation satisfies 
			\begin{align*}
				|\mu_{f}(z)|&=\frac{|f_{\overline{z}}(z)|}{|f_{z}(z)|}=\frac{\left|\overline{g'(z)}+\frac{A}{2(\overline{z}-p)}+\frac{Ap}{2(1-p\overline{z})}\right|}{\left|-\frac{1}{(z-p)^{2}}+h'(z)+\frac{A}{2(z-p)}+\frac{Ap}{2(1-pz)}\right|}\\
				&\le\frac{|g'(z)|+\frac{|A|}{2|z-p|}+\frac{|A|p}{2(1-p)}}{\frac{1}{|z-p|}\left(\frac{1}{|z-p|}-\frac{|A|}{2}\right)-\frac{|A|p}{2(1-p)}-|h'(z)|}\le k.
			\end{align*}
			Therefore, $f$ is $k$-q.c. in $\mathbb{D}$.
			
			It is easy to see that the extension $F$ of $f$ in \eqref{107} coincides with $f$ on $\partial\mathbb{D}$, and hence, $F$ is continuous in $\hat{\mathbb{C}}$. 
			Finally, we prove that $F$ is sense-preserving in $\mathbb{D}^{*}$. Denoting $F(1/\zeta)$ as $\tilde{F}(\zeta)$, then from \eqref{107} we have 
			\begin{equation*}
				\tilde{F}(\zeta)=\frac{\zeta}{1-p\zeta}+h(\overline{\zeta})+\overline{g(\overline{\zeta})},\quad \zeta\in \mathbb{D}.
			\end{equation*}
			Using \eqref{106} again we have
			\begin{align*}
				&|\tilde{F}_{\zeta}|=\left|\frac{1}{(1-p\zeta)^{2}}+\overline{g'(\overline{\zeta})}\right|\ge \frac{1}{(1+p)^{2}}-|g'(\overline{\zeta})|,\quad \text{and}\\
				&|\tilde{F}_{\overline{\zeta}}|=\left|h'(\overline{\zeta})\right|\le \frac{k}{(1+p)^{2}}-|g'(\overline{\zeta})|,
			\end{align*}
			which gives $J_{F}(1/\zeta)|\zeta|^{-4}=J_{\tilde{F}}(\zeta)=|\tilde{F}_{\zeta}|^{2}-|\tilde{F}_{\overline{\zeta}}|^{2}>0$.
			
			Then we proceed to show that $F$ is a sense-preserving homeomorphism of $\widehat{\mathbb{C}}$.
			Since $J_{F} = |F_{z}|^{2} - |F_{\bar{z}}|^{2} > 0$
			wherever it is finite, $F$ is sense-preserving on $\widehat{\mathbb{C}}$.
			Moreover, $F$ agrees with $f$ on $\mathbb{D}$, and by construction it is continuous on $\widehat{\mathbb{C}}$. Since $f$ is univalent in $\mathbb{D}$ and $f(\partial \mathbb{D})$ is a Jordan curve, it follows from the argument principle for sense-preserving harmonic mappings that $F$ is injective in $\mathbb{D}^{*}$ and maps $\mathbb{D}^{*}$ onto the complementary component of $f(\mathbb{D})$.
			Hence, $F : \widehat{\mathbb{C}} \to \widehat{\mathbb{C}}$ is a sense-preserving bijective continuous mapping. Since $\widehat{\mathbb{C}}$ is both compact and $T_{2}$, it follows that $F$ is a homeomorphism (see [\citenum{munkres}, Thm.26.6]).
			At last, we show that $|\mu_{f}|\le k$ in $\mathbb{D}^{*}$. For $\zeta=1/z\in \mathbb{D}$, we get 
			\begin{equation*}
				|\mu_{f}(z)|=|\mu_{\tilde{F}}(\zeta)|=\frac{|\tilde{F}_{\overline{\zeta}}(\zeta)|}{|\tilde{F}_{\zeta}(\zeta)|}\le \frac{\frac{k}{(1+p)^{2}}-|g'(\overline{\zeta})|}{\frac{1}{(1+p)^{2}}-|g'(\overline{\zeta})|}<k.
			\end{equation*}
			The proof is complete.
		\end{proof}

        By using this theorem, we can construct some examples of functions in $\Sigma_{H}^{k}(p)$.
        
        \begin{example}  For $p\in[0,1)$, $k\in[0,1)$ and $A\in \mathbb{C}$ with $|A|<2k/(1-p^2)$, define
            \begin{align*}
                f_{\theta}(z)=\frac{1}{z-p}-c\cdot\operatorname{Re}\left(\frac{e^{i\theta}z}{1-pe^{i\theta}z}\right)+A\log\left|\frac{1-pz}{z-p}\right|,\quad z\in \mathbb{D},\ \theta\in [0,2\pi),
            \end{align*}
            where 
            \begin{align*}
                c=\left(\frac{1-p}{1+p}\right)^{2}k-\frac{(1+p^2)(1-p)}{1+p}|A|>0
            \end{align*}
            by assumption.
            A simple calculation shows that this function satisfies the condition of Theorem \ref{extensiontheorem}:
            \begin{align*}
                &h(z)=g(z)=-\frac{ce^{i\theta}z}{2(1-pe^{i\theta}z)},\\
                &|h'(z)|+|g'(z)|=\frac{c}{|1-pe^{i\theta}z|^{2}}\le\frac{c}{(1-p)^{2}}
                =\frac{k}{(1+p)^{2}}-\frac{1+p^{2}}{1-p^{2}}|A|.
            \end{align*}
           Thus $f_{\theta}(z)$ admits a $k-$q.c. extension to $\widehat{\mathbb{C}}$ and the extension can be expressed as
            \begin{align*}
                F_{\theta}(z)=\begin{cases}
                    f_{\theta}(z), &z\in\mathbb{D},\\
                    \dfrac{1}{z-p}-c\operatorname{Re}\dfrac{1}{ze^{i\theta}-p}, &z\in \overline{\mathbb{D}^{*}}.
                \end{cases}
            \end{align*}
        \end{example}

		\bibliographystyle{plain} 
		\bibliography{main}

	\end{document}